\documentclass[book, 11pt]{ip-journal-jdg}

\usepackage{amsfonts,amsmath, amssymb, mathdots}

\usepackage{amssymb}
\usepackage{amsthm}
\usepackage{amsfonts}
\usepackage{amsmath}

\newtheorem{definition}{Definition}
\newtheorem{proposition}{Proposition}
\newtheorem{theorem}{Theorem}
\newtheorem{lemma}{Lemma}
\newcommand\tr{\mathrm{tr}\,}

\newcommand\goth{\mathfrak}
\newcommand\R{\mathbb R}

\newcommand{\so}{\mathfrak{so}}
\newcommand{\OO}{\mathrm{O}}
\newcommand{\SO}{\mathrm{SO}}
\newcommand{\gl}{\mathfrak{gl}}
\newcommand{\GL}{\mathrm{GL}}

\begin{document}

\title{On one class of holonomy groups\,\\in pseudo-Riemannian geometry}
\author{Alexey~Bolsinov and Dragomir~Tsonev}

\maketitle

\begin{abstract}
\tiny {We describe a new class of holonomy groups on pseudo-Riemannian manifolds. Namely, we prove the following theorem. Let $g$ be a nondegenerate bilinear form of arbitrary signature on a vector space $V$ and  $L:V \to  V$ a $g$-symmetric operator.  Then the identity component of the centraliser of $L$ in the pseudo-orthogonal group $\OO(g)$ is a holonomy group for a suitable Levi-Civita connection.}

\end{abstract}

\section{Introduction and main result}

Holonomy groups were introduced by \'Elie Cartan in the twenties  \cite{Ca1, Ca2} for the study of Riemannian symmetric spaces  and since then the classification of holonomy groups has remained one of the classical problems in differential geometry.

\begin{definition}
Let $M$ be a smooth manifold endowed with an affine symmetric connection $\nabla$.  Then the  holonomy group of $\nabla$ is defined as the subgroup
$\mathrm{Hol}(\nabla) \subset \GL(T_x M)$ that consists of the linear operators $A:T_xM\to T_xM$ being  parallel transport transformations
along closed loops $\gamma$ with $\gamma(0)=\gamma(1)=x$.
\end{definition}

\noindent {\bf Problem.} Can a given subgroup  $H\subset \GL(n,\R)$ be realised as the holonomy group for an appropriate symmetric connection on $M^n$?

\medskip

The fundamental results in this direction are due to Marcel Berger \cite{Ber} who initiated the programme of classification of Riemannian  and irreducible holonomy groups which was  completed by  
D.~V.~Alekseevskii \cite{Alexeevski}, 
R.~Bryant \cite{Br1, Br2},  D.~Joyce \cite{Joy1, Joy2, Joy3},   S.~Merkulov, L.~Schwachh\"ofer \cite{MerkSchw} and S.\,T.~Yau \cite{Yau}.  Very good historical surveys can be found in \cite{Br4, Sch2}.

The classification of Lorentzian holonomy groups has recently been obtained by T.~Leistner \cite{Le1} and A.~Galaev \cite{Gal2}. However, in the general pseudo-Riemanian case,  the complete  description of holonomy groups is a very difficult problem which still remains open,  and even particular examples  are of interest (see \cite{Bergery,  Gal1, Gal3, Ike}). We refer to \cite{GalaevLeistner}  for more information on recent development in this field.

In our paper, we deal with Levi-Civita connections only. In algebraic terms this means that we consider only subgroups of the (pseudo) orthogonal group $\OO(g)$:
$$
H \subset  \OO(g)=\{ A\in \GL(V)~|~  g(Au,Av)=g(u,v), \  u,v\in V\},
$$
where $g$ is a non-degenerate bilinear form on $V$. 

The main result of the present paper is the following theorem. 

\begin{theorem}
\label{main}
For every $g$-symmetric operator $L:V\to V$, the identity connected component $G^0_L$ of  its centraliser   in $\OO(g)$
$$
G_L=\{ X\in \OO(g)~|~ XL=LX\}
$$
is a holonomy group for a certain (pseudo)-Riemannian metric $g$.
\end{theorem}

Notice that in the Riemannian case this theorem becomes trivial:  $L$ is diagonalisable and the connected component  $G^0_L$ of its centraliser  is isomorphic to the standard direct product $\SO(k_1)\oplus \dots \oplus \SO(k_m) \subset  \SO(n)$, $\sum k_i\le n$,  which is, of course,  a holonomy group.  In the pseudo-Riemannian case,  $L$ may have non-trivial Jordan blocks and the structure of $G^0_L$ becomes more complicated.

The structure of the paper is as follows.  First  we recall in Section~\ref{basics} the classical approach by Berger to studying holonomy groups. Like many other authors, we are going to use this approach in our paper. However,  in our opinion, the most interesting part of the present work consists in two explicit matrix formulas \eqref{r} and \eqref{r2}  that, in essence,  almost immediately lead to the solution. To the best of our knowledge, this kind of formulas did not appear in the context of holonomy groups before, and we would really appreciate any comments on this matter.  They came to ``holonomy groups''  from ``integrable systems on Lie algebras'' via ``projectively equivalent metrics''  and we explain this passage in Section 3.  The proof itself is given in Sections 4 (algebraic reduction), 5 (Berger test) and 6 (geometric realisation).   The last section (locnf, 
) contains some details related to the special case when $L$ has complex conjugate eigenvalues.

{\it Acknowledgements}.  We would like to thank D.~Alekseevskii, V.~Cort\'es, E.~Ferapontov, V.~Matveev and  T.~Leistner  for useful discussions. We  are also very grateful to the referee for the valuable remarks that helped us to improve substantially the structure of the paper.



\section{Some basic facts about holonomy groups:   \\ Ambrose-Singer theorem and Berger test}
\label{basics}

Let $\gamma$ be a curve connecting two points $x,y\in M$ (we think of $x$ as a fixed reference point while $y$ is variable) and $P_\gamma: T_xM \to T_y M$ denotes the parallel transport transformation.  The holonomy groups $\mathrm{Hol}_x\,(\nabla)$ and 
$\mathrm{Hol}_y\,(\nabla)$ related to these points are obviously conjugate by means of $P_\gamma$, i.e.,
$$
\mathrm{Hol}_y\,(\nabla) = P_\gamma \circ \mathrm{Hol}_x\,(\nabla) \circ P_\gamma^{-1}.
$$
In particular,  if $M$ is  connected, then the holonomy groups at different points are isomorphic.  

Notice that  $P_\gamma$ allows us to ``transfer''  from $x$ to $y$ (or back from $y$ to $x$) not only tangent vectors, but also tensors of any type.  For example, if $R : \Lambda^2 (T_yM) \to \goth{gl}(T_y M)$ is the curvature tensor of $\nabla$ at the point $y$,   then at then point $x$ we can define the transported tensor $R_\gamma: \Lambda^2 (T_xM) \to \goth{gl}(T_x M)$ as
$$
R_\gamma (u\wedge v)  = P_\gamma^{-1} \circ R\bigl(P_\gamma (u)\wedge P_\gamma(v)\bigr) \circ P_\gamma, \qquad u,v\in T_x M.
$$

The famous  Ambrose-Singer theorem \cite{Ambrose} gives the following description of the Lie algebra $\goth{hol}\,(\nabla)$ of the holonomy group $\mathrm{Hol}\,(\nabla)=\mathrm{Hol}_x\,(\nabla)$ in terms of the curvature tensor $R$:
\medskip

\noindent {\it  $\goth{hol}\,(\nabla)$  is generated (as a vector space) by the operators of the form $R_\gamma (u\wedge v)$.}

\medskip

This motivates the following construction.

\begin{definition}  A map $R: \Lambda^2V \to \gl(V)$ is called a  formal curvature tensor if it satisfies the Bianchi identity
\begin{equation}
R(u\wedge v)w + R(v\wedge w)u + R(w \wedge u)v = 0 \qquad \mbox{for all} \ u,v,w\in V.
\end{equation}
\end{definition}

This definition simply means that $R$ as a tensor of type $(1,3)$  satisfies all usual algebraic properties of  curvature tensors:
$$
R_{k\, ij}^m=R_{k\, ji}^m \quad \mbox{and}  \quad R_{k\, ij}^m+R_{i\, jk}^m+R_{j \, ki}^m=0.
$$

\begin{definition}\label{defBerger}  Let $\goth h\subset \gl(V)$ be a Lie subalgebra.  Consider the set of all formal curvature tensors $R: \Lambda^2V \to \gl(V)$ such that $\mathrm{Im}\, R\subset \goth h$:
$$
\mathcal R(\goth h)=\{ R: \Lambda^2V \to \goth h ~|~
R(u\wedge v)w + R(v\wedge w)u + R(w \wedge u)v = 0, \   u,v,w\in V\}.
$$
We say that $\goth h$ is a Berger algebra if it is generated as a vector space by the images of the formal curvature tensors $R\in \mathcal R(\goth h)$, i.e.,
$$
\goth h = \mathrm{span} \{ R(u\wedge v)~|~ R\in \mathcal R(\goth h), \ u,v\in V\}.
$$
\end{definition}

{\it Berger's test}  (sometimes referred to as Berger's criterion) is the following result which can, in fact,  be viewed as a version of the Ambrose--Singer theorem:

\medskip
{\it Let $\nabla$ be a  symmetric affine connection on $TM$.  Then the Lie algebra $\goth{hol} \, (\nabla)$ of its holonomy group $\mathrm{Hol} \, (\nabla)$ is {\it Berger}.}
\medskip

Usually the solution of the classification problem for holonomy groups consists of two parts. At first, one tries to describe all Lie subalgebras $\goth h\subset \gl(n,\R)$ of a certain type satisfying Berger's test  (i.e., Berger algebras). This part is purely algebraic. The second (geometric) part is to find
a suitable connection $\nabla$ for a given Berger algebra $\goth h$ which realises  $\goth h$ as the holonomy Lie algebra, i.e., $\goth h=\goth{hol}\, (\nabla)$.

 We follow the same scheme but use, in addition,  some ideas from two other areas of mathematics:  projectively equivalent metrics and  integrable systems on Lie algebras. These ideas are explained in  the next section.
The reader who is interested only in the proof itself may  proceed directly to  Sections 4, 5 and 6 which are formally independent of this preliminary discussion.


\section{Projectively equivalent metrics and sectional operators}\label{sectional}

The problem we are dealing with is closely related to the theory of projectively equivalent (pseudo)-Riemannian metrics  \cite{Aminova1,  Aminova2,  LC,  Matveev,   Sinjukov,  Mikes}.

\begin{definition} Two metrics $g$ and $\bar g$ on a manifold $M$ are called projectively equivalent,  if they have the same geodesics considered as unparametrised curves.
\end{definition}

In the Riemannian case the local classification of projectively equivalent pairs $g$ and $\bar g$ was obtained by Levi-Civita in 1896 \cite{LC}.  For  pseudo-Riemannian metrics, a complete description in reasonable terms of all possible projectively equivalent pairs is still an open problem  although it has been intensively studied (see \cite{Aminova1, Aminovabook, splitting, locnf, KM1, KM2, Solodovnikov}) and many particular examples and results in this direction have been obtained.

As a particular case of projectively equivalent metrics $g$ and $\bar g$ one can distinguish the following, which is closely related to our problem. Assume that $g$ admits a covariantly constant $(1,1)$-tensor field $L$. Then we can introduce a new metric $\bar g$ by setting
$$
\bar g(\xi, \eta) = g(L\xi, \eta).  
$$
Now, the metrics $g$ and $\bar g$ are not only projectively equivalent, but their geodesics coincide as parametrised curves.  In this case $g$ and $\bar g$ are called {\it affinely equivalent}.

In the pseudo-Riemannian case, the classification of pairs $(g, L)$ such that $\nabla L=0$   is an interesting  problem,  which has been solved only partially (see, for example, \cite{Solodovnikov}).

The condition $\nabla L=0$ can be interpreted in terms of the holonomy group  $\mathrm{Hol}(\nabla)$. Namely,
$g$ admits a covariantly constant $(1,1)$-tensor field if and only if
$\mathrm{Hol}(\nabla)$  is  a subgroup of the centraliser of $L$ in $\OO(g)$:
$$
\mathrm{Hol}(\nabla)\subset G_L = \{ X\in \OO(g)~|~  XLX^{-1}=L\}.
$$
Here by $L$ we understand the value of the desired $(1,1)$-tensor field at a fixed point $x_0\in M$.  The field  $L(x)$ itself can be constructed from the initial condition $L=L(x_0)$ by a parallel transport.  The independence of the choice of  a path $\gamma$ between $x_0$ and $x$ is automatically guaranteed by the inclusion $\mathrm{Hol}(\nabla)\subset G_L$.

In the spirit of the present discussion it is natural to conjecture that for a generic metric $g$ satisfying $\nabla L=0$, its holonomy group coincides with $G^0_L$, which is basically the statement of our main theorem.  

Since we are going to use Berger's approach, the role of the curvature tensor will be very important.  Our proof will be based on one unexpected and remarkable relationship between the algebraic structure of the curvature tensor of projectively equivalent metrics and integrable Hamiltonian systems on Lie algebras.

To explain this relationship, 
we first notice that  $\Lambda^2 V$  can be naturally identified with $\so(g)$.  Therefore, in the (pseudo)-Riemannian case,  a curvature tensor at a fixed point can be understood as a linear map
$$
R: \so(g) \to \so(g).
$$

Some operators of this kind play  important role in the theory  of integrable systems on semisimple Lie algebras.
 
\begin{definition}   We say that a linear map
$$
R:  \so(n) \to \so(n)
$$
is  a  sectional operator, if $R$ is self-adjoint w.r.t. the Killing form and satisfies the algebraic identity:
\begin{equation}
[R(X), L] = [X, M] \quad \mbox{for all } X\in  \so(n),
\label{sect1}
\end{equation}
where $L$ and $M$ are some fixed symmetric matrices.
\end{definition}

These operators first appeared in the famous paper by S.~Manakov \cite{Manakov} on the integrability of a multidimensional rigid body and then were studied by A.~Mischenko and A.~Fomenko in the framework of the argument shift method  \cite{MischFom}. 
The terminology ``sectional'' was suggested by A.~Fomenko and V.~Trofimov
 \cite{TrofFom} for a more general class of operators on Lie algebras with similar properties and originally was in no way related to ``sectional curvature''.  However, such a relation exists and is, in fact, very close.

The following observation, which is, in fact,  an algebraic interpretation of the so-called second Sinjukov equation \cite{Sinjukov} for projectively equivalent metrics,  was made in \cite{BolsMatvKios}.
 
\begin{theorem}
\label{BMK}
If $g$ and $\bar g$ are projectively equivalent, then the curvature tensor of $g$ considered as a linear map
$R:  \so(g) \to \so(g)$
is a sectional operator, i.e., satisfies identity \eqref{sect1}
with $L$ defined by $\bar g^{-1} g = \det L \cdot L$
and $M$ being the Hessian of $2\tr L$, i.e. $M^i_j=2\nabla^i\nabla_j \tr L$.
\end{theorem}

It turns out that there is an elegant explicit formula expressing $R(X)$ in terms of $L$ and $M$.  To get this formula, one first needs to notice that \eqref{sect1} immediately implies that $M$ belongs to the centre of the centraliser of $L$ and, therefore, can be written as $M=p(L)$ where $p(t)$ is a certain polynomial. Then
\begin{equation}
\label{curv}
R(X) = \frac{d}{dt}\big|_{t=0} p(L + tX)
\end{equation}
satisfies \eqref{sect1}.
To check this, it is sufficient to differentiate the identity
$$
[p(L+tX), L+tX]=0
$$
with  respect to $t$ to get
$$
[\frac{d}{dt}\big|_{t=0} p(L + tX), L] + [p(L), X] =0
$$
i.e., $[R(X), L] + [M,X]=0$ as needed.

In the case of affinely equivalent metrics (we are going to deal with this case only!),  $L$ is automatically covariantly constant and, therefore, $M=0$. Thus, the curvature tensor  $R$  satisfies a simpler equation
$$
[R(X), L] =0,
$$
which, of course, directly follows from $\nabla L=0$ and seems to make all the discussion above not relevant to our very particular situation. 
However,   formula \eqref{curv}   still defines a non-trivial operator, if  $p(t)$ is a non-trivial  polynomial satisfying $p(L)=M=0$, for example the minimal polynomial for  $L$.

This discussion gives us a very good candidate for the role of a formal curvature tensor in our construction, namely,   the operator defined by \eqref{curv} with $p(t)$ being the minimal polynomial of $L$. As we show below, this operator satisfies all the required conditions and this observation plays a crucial role in the proof of Theorem~\ref{main} which is given in the next three sections.


\section{Step one. Reduction to the case of a single real eigenvalue or a~pair of complex conjugate eigenvalues}\label{reduction}

Let $g$ be a non-degenerate bilinear form on $V$ and  $L:V\to V$ be a $g$-symmetric operator. First of all, we notice that it is sufficient to prove Theorem \ref{main} for two special cases only:
\begin{itemize}
\item either $L$ has a single real eigenvalue,
\item or $L$ has a pair of complex conjugate eigenvalues.
\end{itemize}

The reduction from the general case to one of these is standard.  If $L$ has several eigenvalues, then $V$ can be decomposed into $L$-invariant subspaces
$$
V= V_1 \oplus V_2 \oplus \dots \oplus V_s
$$
where $V_i$ is either a generalised eigensubspace corresponding to a real eigenvalue $\lambda_i$, or  a similar subspace corresponding to a pair of complex conjugate eigenvalues $\lambda_i$ and $\bar \lambda_i$.  

This decomposition is $g$-orthogonal.  Indeed, assume that $V_i$ is related to a real eigenvalue $\lambda_i$ and consider the operator $(L - \lambda_i)^k$, where $k\in \mathbb N$ is sufficiently large. 
Then each $V_j$ is invariant with respect to this operator.   Moreover, we have  $(L - \lambda_i)^k V_i = 0$ and $(L - \lambda_i)^k V_j = V_j$ for $i\ne j$.  Now, since $L$ is $g$-symmetric we get
$$
0 = g\bigl((L - \lambda_i)^k V_i , V_j\bigr)= g \bigl(V_i, (L - \lambda_i)^k V_j\bigr) = g\bigl(V_i,  V_j\bigr),
$$
as needed. 

If $V_i$ corresponds to a pair of complex conjugate eigenvalues $\lambda_i$ and $\bar\lambda_i$,  then the same argument is applied to the operator $\bigl( (L - \lambda_i)(L-\bar\lambda_i)\bigr)^k$.

Furthermore,  the group $G^0_L$ is compatible with this decomposition in the sense that
$G^0_L$ is the direct product of the Lie groups $G_1,\dots, G_s$  each of which is naturally associated with $V_i$ and is the connected component of the centraliser of $L|_{V_i}$ in $\OO (g|_{V_i})$.  This follows immediately from the fact that every operator $X$ commuting with $L$ leaves  the generalised eigenspaces $V_j$'s invariant.

Thus,  $G^0_L$ is reducible, and therefore  $G^0_L$ is a holonomy group if and only if so is each $G_i$.  For our purposes we need  a weaker version of this statement:  if each $G_i$ is a holonomy group, then so is $G^0_L$.  The explanation of this fact is very simple. If we  can realise each $G_i$ as a holonomy group on a certain pseudo-Riemannian manifold $M_i$, then  the holonomy group for the (pseudo)-Riemannian direct product manifold  $M=M_1\times \dots\times M_s$ will be exactly  $G^0_L = G_1 \times \dots \times  G_s$.  A similar reduction obviously takes place for the corresponding Lie algebras. Indeed, the Lie algebra $\goth g_L$ of $G^0_L$  splits into the direct sum $\goth g_1\oplus \dots \oplus \goth g_s$,  and $\goth g_L$ is Berger if and only if so is each $\goth g_i$, $i=1,\dots, s$.

Thus, from now on we assume that the $g$-symmetric operator $L$ has either a single real eigenvalue or two complex conjugate eigenvalues. For the matrix computations in the next section, we will need the explicit form of $\goth g_L$ in this reduced case.  In fact the case of a pair of complex conjugate eigenvalues is quite similar to the real case.  Below we concentrate on the case when $\lambda\in\mathbb R$ and  all the necessary amendments related to the complex situation will be discussed in Appendix.

In the real case, we use the following well-known analog of the Jordan normal form theorem for $g$-symmetric operators in the case when $g$ is pseudo-Euclidean (see, for example, \cite{Lancaster, Thompson}).

\begin{proposition}\label{Jordan}
Let $L:V\to V$ be a $g$-symmetric operator with a single eigenvalue $\lambda\in\mathbb R$. Then by an appropriate choice of a basis in $V$,   we can simultaneously reduce $L$ and $g$ to the following block diagonal matrix form:
\begin{equation}
\label{Lg}
L=\begin{pmatrix}
L_1 &  & & \\
& L_2  & & \\
& & \ddots & \\
& & & L_k
\end{pmatrix}, \qquad
g=\begin{pmatrix}
g_1 &  & & \\
& g_2  & & \\
& & \ddots & \\
& & & g_k
\end{pmatrix}
\end{equation}
where 
$$
L_i=\begin{pmatrix}
\lambda & 1 & & &\\
& \lambda & 1 &  & \\
& & \ddots &\ddots & \\
   &  & & \lambda & 1\\
   & & & &  \lambda
   \end{pmatrix}, \quad  \mbox{and} \quad   g_i=\pm\begin{pmatrix}
 &  & & & 1\\
   &  &  & 1 & \\
   &  & \iddots & &\\
   & 1& &  &  \\
  1 & & & &
   \end{pmatrix}.
   $$
The blocks $L_i$ and $g_i$ are of the same size  $n_i\times n_i$,  and
   $n_1\le n_2 \le \dots \le n_k$.  
As a particular case, we admit $1\times 1$ blocks
   $L_i=\lambda$ and $g_i=\pm 1$.
 \end{proposition}

In this proposition and below,  we use the same notation $L$  for the operator and its matrix.
This does not lead to any confusion because from now on we can choose and fix a canonical basis. The same convention is applied to the bilinear form $g$ and its matrix.
In what follows,  we shall assume that $g_i$  has  $+1$ on the antidiagonal.  This assumption is not very important,  but  allows us to simplify the formulae below.

The next statement  gives an explicit matrix description for $\goth{so}(g)$ and the Lie algebra $\goth g_L$  of the group $G^0_L$   for $L$ and $g$ described in Proposition~\ref{Jordan}. The proof is straightforward and we omit it.

\begin{proposition}
\label{description}
In the canonical basis from Proposition \ref{Jordan},  the orthogonal Lie algebra $\so(g)$ consists of block matrices of the form
\begin{equation}
\label{X}
X =
\begin{pmatrix}
  X_{11}& X_{12} & \cdots & X_{1k}  \\
 X_{21} & X_{22}  & \cdots & X_{2k}  \\
 \vdots & \vdots & \ddots & \vdots \\
 X_{k1} & X_{k2} & \cdots & X_{kk} \\
 \end{pmatrix}
\end{equation}
where $X_{ij}$ is an $n_i\times n_j$ block.
The diagonal blocks $X_{ii}$'s are skew-symmetric with respect to their antidiagonal.  The
off-diagonal blocks $X_{ij}$ and $X_{ji}$  are related by
$$
X_{ji} = - g_j  X_{ij}^\top g_i.
$$
More explicitly:
 \begin{equation}
 \label{Xij}
X_{ij} =
\begin{pmatrix}
   x_{11} & \cdots & x_{1n_j}  \\
   \vdots & \ddots & \vdots    \\
    x_{n_i1}& \cdots & x_{n_i n_j}  \\
 \end{pmatrix} , \quad
 X_{ji} =
\begin{pmatrix}
   -x_{n_in_j} & \cdots & -x_{1n_j}  \\
   \vdots & \ddots & \vdots    \\
    -x_{n_i1}& \cdots & -x_{11}  \\
\end{pmatrix}
\end{equation}

The Lie algebra $\goth g_L$ consists of block matrices of the form:
\begin{equation}
\left( {\begin{array}{cccc}
  0 & M_{12} & \cdots & \!\!\!\!\!\! M_{1k}  \\
 M_{21} & 0  &  &  \!\!\!\!\!\! \vdots  \\
 \vdots &  & \ddots &  \! M_{k-1,k} \\
 M_{k1} & \cdots & M_{k, k-1} &  \!\!\!\!\!\!\!\!\! 0\\
 \end{array} } \right)
\label{m1}
\end{equation}
where $M_{ij}$'s for $i<j$ are  $n_i\times n_j$ matrices of the form
\begin{equation}
\label{Mij}
M_{ij} = \begin{pmatrix}
   0          & \cdots & 0 & \mu_1 & \mu_{2} &  \cdots &  \mu_{n_i}  \\
    0         & \cdots & 0 &      0       & \mu_{1} &  \ddots &  \vdots     \\
    \vdots & \ddots & \vdots & \vdots &  \ddots &  \ddots   & \mu_2 \\
    0         & \cdots  & 0        & 0        & \cdots & 0 &  \mu_1
\end{pmatrix}, \  \mu_i\in \R.
\end{equation}
If $n_i=n_j$, then $M_{ij}$ is a square matrix and the first zero columns are absent.  The blocks $M_{ij}$ and $M_{ji}$ are related in the same way as  $X_{ij}$ and $X_{ji}$, i.e.,     $ M_{ji} = - g_j  M_{ij}^\top g_i$.

   The subspace $\goth m_{ij} \subset \goth g_L$   $(i < j)$ that consists of two blocks $M_{ij}$ and $M_{ji}$ is a commutative subalgebra of dimension $n_i$.  As a vector space,   $\goth g_L$ is the direct sum $\sum_{i<j} \goth m_{ij}$.  In particular, $\dim \goth g_L = \sum_{i=1}^k (k-i) n_i$.
\end{proposition}

Our next goal is to prove that the (matrix) Lie algebra $\goth g_L$ described in this Proposition is Berger.


\section{Step two: Berger's  test}\label{Berger}

We consider a non-degenerate bilinear form $g$ on a finite-dimensional real vector space $V$,  and a $g$-symmetric linear operator $L: V\to V$, i.e.,
$$
g(Lv, u)= g(v, Lu), \quad \mbox{for all } u,v\in V.
$$

As before, by $\so(g)$ we denote the Lie algebra of the orthogonal group associated with $g$. Recall that this Lie algebra consists  of $g$-skew-symmetric operators:
$$
\so(g)=\{  X:V\to V~|~ g(Xv, u)=-g(v, Xu), \quad u,v\in V\}.
$$

Consider the Lie algebra $\goth g_L$ of the group $G^0_L$: 
$$
\goth g_L =  \{ X\in \so(g)~|~  XL-LX=0\}.
$$
We are going to verify in this section that $\goth g_L$ is a Berger algebra.

Following  Definition~\ref{defBerger},  we need to describe the formal curvature tensors  $R: \Lambda^2 V \to \goth g_L$ and analyse the subspace in $\goth g_L$ spanned by their images.  In particular, if we can find just one single formal curvature tensor $R$ such that
$\mathrm{Im}\, R = \goth g_L$, then our goal will be achieved.

In general, this is a rather difficult problem because the Bianchi identity represents a highly non-trivial system of linear relations.  However,  as was explained in Section \ref{sectional},  we have a very good candidate for the role of $R$.

In what follows,  we use the following natural identification of $\Lambda^2 V$ and $\so(g)$:
$$
 \Lambda^2 V \longleftrightarrow so(g),      \quad    v\wedge u = v\otimes g(u) - u \otimes g(v).
$$
Here the bilinear form $g$ is understood as an isomorphism $g:V \to V^*$  between vectors and covectors.
Taking into account this identification, we define the linear mapping $R: \so(g)\simeq \Lambda^2V \to \gl(V)$
by:
 \begin{equation}
 \label{r}
 R(X)= \frac{d}{dt} \big| _{t=0} p_{\mathrm{min}}(L+tX),
 \end{equation}
where $p_{\mathrm{min}}(t)$ is the minimal polynomial of $L$.

\begin{proposition} Let $L:V \to V$ be a $g$-symmetric operator.
Then  \eqref{r}  defines a formal curvature tensor   $R: \Lambda^2V \simeq \so(g) \to \goth g_L$  for the Lie algebra $\goth g_L$.
In other words, $R$ satisfies the Bianchi identity and its image is contained in $\goth g_L$.
\end{proposition}

The proof consists of two lemmas.

\begin{lemma}
\label{lem1}
The image of $R$ is contained in $\mathfrak {g}_L $.
\end{lemma}

\begin{proof}

First we check that  $R(X) \in \so(g)$, i.e.,   $R(X)^* = -R(X)$, where  $*$ denotes  ``$g$--adjoint'':
$$
g(A^*u,v)=g(u, Av), \qquad  u,v\in V.
$$

Since  $L^{\ast}=L$, $X^{\ast}=-X$,
$(p_{\mathrm{min}}(L+tX))^{\ast}=p_{\mathrm{min}}(L^{\ast}+tX^{\ast})$ and $"\frac {d}{dt}"$ and $"\ast"$ commute, we have
$$
\begin{aligned}
R(X)^* & = \frac{d}{dt}\big|_{t=0} p_{\mathrm{min}}(L+tX)^*  =
\frac{d}{dt}\big|_{t=0} p_{\mathrm{min}}(L^*+tX^*)= \\
& =\frac{d}{dt}\big|_{t=0} p_{\mathrm{min}}(L-tX)=
-\frac{d}{dt}\big|_{t=0} p_{\mathrm{min}}(L+tX)= -R(X),
\end{aligned}
$$
as needed.  Thus, $R(X) \in \so(g)$. Notice that this fact holds true for any polynomial $p(t)$, not necessarily minimal.

To prove that $R(X)$ commutes with $L$, we
consider the obvious identity
$$
[p_{\mathrm{min}}(L+tX),L+tX]=0.
$$
and differentiate it at $t=0$:
$$
[\frac{d}{dt}\big|_{t=0}p_{\mathrm{min}} (L+tX),L]+[p_{\mathrm{min}}(L),X]=0.
$$
Now, $p_{\mathrm{min}}(L)=0$ as it is a minimal polynomial.
Hence $[R(X),L]=0$, and therefore  $R(X) \in \goth {g}_L$. \end{proof}

\begin{lemma}
\label{lem2}
$R$ satisfies the Bianchi identity, i.e.
$$
R(u\wedge v)w+R(v\wedge w)u+R(w\wedge u)v=0\quad \mbox{ for all } u,v,w \in V.
$$
\end{lemma}

\begin{proof}
It is easy to see that our operator  $R: \Lambda^2 V\simeq \so(g) \to \gl(V)$  can be written as $R(X)=\sum_k C_kXD_k$, where $C_k$ and $D_k$ are some $g$-symmetric operators  (in our case these operators are some powers of $L$).  Thus, it is sufficient to check the Bianchi identity for operators of the form $X \mapsto CXD$.  

For $X=u\wedge v$ we have
$$
C(u\wedge v)Dw=  Cu \cdot g(v, Dw) - Cv \cdot g(u, Dw)
$$
Similarly, if we cyclically permute $u,v$ and $w$:
\begin{equation*}
C(v\wedge w)Du=  Cv \cdot g(w, Du) - Cw \cdot g(v, Du)
\end{equation*}
and
\begin{equation*}
C(w\wedge u)Dw=  Cw \cdot g(u, Dv) - Cu \cdot g(w, Dv).
\end{equation*}
Adding these three expressions and  taking into account that both $C$ and $D$ are $g$-symmetric, we obtain zero, as required. \end{proof}

The construction presented above is invariant and and can be applied to any $g$-symmetric operator $L$, in particular, with distinct eigenvalues.  To complete the proof we need to compute the image of   \eqref{r}  and compare it with $\goth g_L$.  We are going to do this by means of matrix linear algebra, and from now on we consider the reduced case with a single real eigenvalue $\lambda\in\mathbb R$   described in Proposition \ref{Jordan} (the case of a pair of complex conjugate eigenvalues is discussed in the Appendix).  Replacing $L$ by $L-\lambda$, we can assume without loss of generality that $\lambda = 0$, i.e., $L$ is nilpotent. 

Proposition \ref{Jordan} implies that in the case of a single Jordan block  the algebra $\goth g_L$ is trivial and thus we begin with the first non-trivial case  when $L$ consists of two Jordan blocks $L_1$ and $L_2$. 


\begin{proposition}
\label{twoblocks}
 Let $L:V \to V$ be a $g$-symmetric nilpotent operator that consists of two Jordan blocks. Then the image of
the formal curvature tensor  $R: \Lambda^2V \simeq \so(g) \to \goth g_L$ defined by \eqref{r} coincides with $\goth g_L$.  In particular, $\goth g_L$ is Berger. 
\end{proposition}

\begin{proof}
We will get this result by straightforward computation in the canonical basis described in Proposition \ref{Jordan}.
Consider $L = \begin{pmatrix} L_1 & 0 \\ 0 & L_2 \end{pmatrix}$, where $L_1$ and $L_2$ are standard nilpotent Jordan blocks of size $m$ and $n$ respectively, $m\le n$.  The minimal polynomial for $L$ is $p_{\mathrm{min}}(t)= t^n$.

If we represent $X\in \so(g)$ as a block matrix  (the sizes of blocks are naturally related to $m$ and $n$)
$$
X = \begin{pmatrix}
X_{11} &  X_{12} \\
X_{21} & X_{22}
\end{pmatrix}
$$
 we then immediately see  that  our operator
\begin{equation*}
R =  \frac{d}{dt} \big| _{t=0} (L+tX)^n=L^{n-1}X+L^{n-2}XL+\cdots + XL^{n-1}
\end{equation*}
acts independently of each block of $X$, i.e.,
 \begin{equation}
 \label{r(x)}
R(X) =
\left( {\begin{array}{cc}
   R_{11}(X_{11}) & R_{12}(X_{12})  \\
    R_{21}(X_{21}) &R_{22}(X_{22}) \\
 \end{array} } \right)
\end{equation}

The blocks  $R_{11}(X_{11})$, $R_{12}(X_{12})$, $R_{21}(X_{21})$ and $R_{22}(X_{22})$ can be explicitly  computed, and we  shall see that  the image of $R$ is exactly our Lie algebra $\goth g_L$.

This computation can be essentially simplified,   if we take into account  the inclusion $\mathrm{Im}\, R \subset \goth g_L$ (Lemma \ref{lem1}) and the fact that  $\goth g_L$ consists of the block matrices of the form  $\begin{pmatrix} 0 & M_{12} \\ M_{21} & 0\end{pmatrix}$, where
\begin{equation}
\label{m}
M_{12}=
\begin{pmatrix}
   0          & \cdots & 0 & \mu_1 & \mu_{2} &  \cdots &  \mu_m  \\
    0         & \cdots & 0 &      0       & \mu_{1} &  \ddots &  \vdots     \\
    \vdots & \ddots & \vdots & \vdots &  \ddots &  \ddots   & \mu_2 \\
    0         & \cdots  & 0        & 0        & \cdots& 0 &  \mu_1
\end{pmatrix}
\end{equation}
 is $m\times n$  matrix, and  $M_{21} = - g_2 M_{12}^\top g_1$  (see Proposition \ref{description}).  Then
 without any computation we can conclude that $R_{11}(X_{11})=0$, $R_{22}(X_{22})=0$ and
 $R_{21}(X_{21})=- g_2 ^\top R_{12}(X_{12}) g_1$.

Thus,  we should only explain how the  parameters
$\mu_1, \dots, \mu_m$ of the block $M=R_{12}(X_{12})$  depend on the entries of
$$
X_{12}=\begin{pmatrix}
x_{11} & x_{12} & \ldots & x_{1n}\\
\vdots &  \vdots & \ddots  & \vdots \\
x_{m1} & x_{k2} & \ldots & x_{mn}
\end{pmatrix}
$$

We have
\begin{equation}
R_{12}(X_{12})=
L_1^{n-1} X_{12} + L_{1}^{n-2}X_{12} L_2 +\cdots + X_{12} L_2^{n-1}
\label{r12}
\end{equation}
and an easy computation gives
$$
\begin{aligned}
\mu_1&=x_{m1} ,\\
\mu_2&=x_{m-1,1}+x_{m2},\\
\mu_3&=x_{m-2,1}+x_{m-1,2}+x_{m3},\\
& \ \vdots \\
\mu_{k}&=x_{11}+x_{22}+ x_{33}+\cdots  +x_{mm}.
\end{aligned}
$$

Clearly,  there are no relations between $\mu_i$'s and therefore the image of $R$ coincides with $\goth g_L$, which completes the proof.\end{proof}

Thus,  formula \eqref{r}  solves the problem in the case of two blocks. Now, let us consider the case of  $k$ Jordan blocks, $k >2$.   In this case,  the image of the formal curvature tensor defined by \eqref{r} can be smaller than $\goth g_L$ and formula \eqref{r} needs to be modified.

We start with the following obvious  remark.
Let $V' \subset V$  be a subspace of $V$  such that $g'= g|_{V'}$ is non-degenerate. Consider  the standard embedding $\so(g') \to \so(g)$  induced by the inclusion $V' \subset V$. 
If
$R': \so(g') \to \so(g')$  is a formal curvature tensor, then
its trivial extension  $R: \so(g)  \to \so(g)$ defined by
$$
R \begin{pmatrix}  X & Y \\ Z & W \end{pmatrix} =
\begin{pmatrix}  R'(X) & 0 \\ 0 & 0 \end{pmatrix}
$$
is a formal curvature tensor too. In particular, if  $\goth h\subset \so(g')$ is a Berger subalgebra, then  $\goth h$ as a subalgebra of  $\so(g)$ will be also Berger.

This remark allows us to construct a ``big'' formal curvature tensor as the sum of ``small'' curvature tensors related to different pairs of Jordan blocks and in this way to reduce the general case to the situation treated in Proposition \ref{twoblocks}.

Consider the operator $\widehat R_{12} : \so(g) \to \so(g)$  defined by:
\begin{equation}
\widehat R_{12}\begin{pmatrix}
X_{11}& \!\!\!\!  X_{12} & \cdots & \!\! X_{1k}  \\
 X_{21} & \!\!\!\!  X_{22}  & \cdots & \!\! X_{2k}  \\
 \vdots & \vdots & \ddots & \vdots \\
 X_{k1} & \!\!\!\!  X_{k2} & \cdots & \!\! X_{kk}
\end{pmatrix}=
\begin{pmatrix}
 0 & R_{12}(X_{12}) &  \cdots & 0  \\
R_{21}( X_{21}) &  0  & \cdots & 0  \\
 \vdots & \vdots & \ddots & \vdots \\
0 & 0 & \cdots &  0
\end{pmatrix}
\label{R12}
\end{equation}
where $R_{12}(X_{12})$ and $R_{21}(X_{21})$ are defined as in Proposition \ref{twoblocks} (see \eqref{r(x)}, \eqref{r12}),  and all the other blocks in the right hand side vanish. Then applying the above remark to the the subspace $V'\subset V$ related to the first two blocks $L_1$ and $L_2$, we see that  $\widehat R_{12}$ is a formal curvature tensor and its image coincides with the Abelian subalgebra $\goth m_{12}\subset \goth g_L$ (see Proposition \ref{description}). In particular, $\goth m_{12}\subset \so(g)$ is a Berger algebra.

To construct the ``big''  formal curvature operator $R: \so(g) \to \goth g_L$  we simply do the same for each pair of blocks,  namely we set:
\begin{equation}
R\begin{pmatrix}
X_{11}& \!\!\!\!  X_{12} & \cdots & \!\! X_{1k}  \\
 X_{21} & \!\!\!\!  X_{22}  & \cdots & \!\! X_{2k}  \\
 \vdots & \vdots & \ddots & \vdots \\
 X_{k1} & \!\!\!\!  X_{k2} & \cdots & \!\! X_{kk}
\end{pmatrix}=
\begin{pmatrix}
 0 & \!\!\!\!\! R_{12}(X_{12}) & \!\!\!\! \cdots & \!\!\!\!\! R_{1k}( X_{1k})  \\
R_{21}( X_{21}) &  \!\!\!\!\! 0  &\!\!\!\!  \cdots & \!\!\!\!\! R_{2k}( X_{2k})  \\
 \vdots & \!\!\!\!\! \vdots &\!\!\!\!  \ddots &\!\!\!\!\!\!\!  \vdots \\
R_{k1}( X_{k1}) & \!\!\!\!\! R_{k2}( X_{k2}) & \!\!\!\! \cdots & \!\!\!\!\!\!\!  0
\end{pmatrix}
\label{rmodified1}
\end{equation}

In other words,  $R$ acts  independently on each block $X_{ij}$  (compare with the proof of Proposition \ref{twoblocks}), and each of its components
$$
R_{ij} :    X_{ij} \mapsto  R_{ij}(X_{ij})
$$
is defined in exactly the same way as in Proposition \ref{twoblocks} provided we ignore all the blocks of $L$ except for $L_i$ and $L_j$.  More precisely,
\begin{equation}
R_{ij} (X_{ij})= L_i^{n_{ij}-1} X_{ij} + L_{i}^{n_{ij}-2}X_{ij}L_j +\cdots + X_{ij}L_j^{n_{ij}-1},
\label{rmodified2}
\end{equation}
where $n_{ij}=\max\{n_i, n_j\}$ and $n_i$, $n_j$ are the sizes of the nilpotent Jordan blocks $L_i$ and $L_j$.  Notice that  the operators \eqref{rmodified1} and \eqref{r} are, in fact,  very similar. The latter also has the same block structure,  and the only difference is that in  \eqref{rmodified2}  instead of $n_{ij}=\max\{n_i, n_j\}$ we need to take the maximum over all $n_1,\dots,n_k$.  This leads to ``widening'' of the kernel of $R$ and ``reducing'' of its image.  To avoid such a situation, we need the modification \eqref{rmodified1}.

If we introduce the operators
$\widehat R_{ij}: \so(g) \to \so(g)$ by generalising \eqref{R12} for arbitrary indices $i<j$,  we can rewrite  \eqref{rmodified1} and \eqref{rmodified2}  as
\begin{equation}
R_{\mathrm{formal}}=R=\sum_{i<j} \widehat R_{ij}.
\label{rmodified3}
\end{equation}

\begin{proposition}
\label{manyblocks}
The operator $R_{\mathrm{formal}}=R$ defined by \eqref{rmodified1} and \eqref{rmodified2}   {\rm(}or equivalently by \eqref{rmodified3}{\rm )} is  a formal curvature tensor.  Moreover,  $\mathrm{Im}\, R = \goth g_L$ and, therefore,
$\goth g_L$ is a Berger algebra.
\end{proposition}

\begin{proof}
Since  each $\widehat R_{ij}$ is a formal curvature tensor,   so is $R$ by linearity.  The image of  $\widehat R_{ij}$ is  the subalgebra $\goth m_{ij}$.    From \eqref{rmodified1} it is easily seen that each $\widehat R_{ij}$ acts only on the blocks $X_{ij}$ and $X_{ji}$ and does not interact with other blocks at all.   This (together with Proposition \ref{description}) immediately implies that
$$
\mathrm{Im}\, R= \sum_{i<j} \mathrm{Im}\, \widehat R_{ij}=
\sum_{i<j} \goth m_{ij} = \goth g_L,
$$
as required. \end{proof}

This proposition tells us that $\goth g_L$ is Berger whenever $L$ has a single real eigenvalue $\lambda\in\R$. In the case of a pair of complex eigenvalues $\lambda$ and $\bar\lambda$, the proof needs just  few additional comments given in the Appendix.   Taking into account the reduction in Section~\ref{reduction}, we arrive at the following final conclusion.

\begin{theorem}
\label{ourBerger}
Let $L:V \to V$ be a $g$-symmetric operator. Then 
 $$
 \goth g_L =\{ X\in \goth{so}(g) ~|~ XL = LX\}
 $$ 
 is a Berger algebra.
\end{theorem}


\section{Step three: Geometric realisation}

Now for a given operator $L:T_{x_0}M \to T_{x_0}M$, we need to find a (pseudo)-Riemannian metric $g$ on $M$ and a $(1,1)$-tensor field $L(x)$  (with the initial condition $L(x_0)=L$) such that
\begin{enumerate}
\item $\nabla L(x)= 0$;
\item $\mathfrak{hol}\,(\nabla)=\goth g_L$.
\end{enumerate}

Notice that the first condition guarantees that $\mathfrak{hol}\,(\nabla) \subset \goth g_L$. On the other hand,
$\mathrm{Im}\, R(x_0) \subset \mathfrak{hol}\,(\nabla)$, where $x_0\in M$ is a fixed point and $R$ is the curvature tensor of $g$. So, taking into account Theorem \ref{ourBerger}, the second condition can be replaced by
\medskip

$2'$) $R(x_0)$ coincides with the formal curvature tensor   $R_{\mathrm{formal}}$ from Proposition \ref{manyblocks}.

\medskip

Thus, our goal in this section is to construct  (at least one example of)  $L(x)$ and $g(x)$ satisfying conditions 1 and $2'$.
Apart  from formula \eqref{r} (whose modification \eqref{r2} leads to the desired example),  the construction below is based on two well-known geometric facts.

The first one allows us to use a nice coordinate system in which all computations at a fixed point become much simpler.  Roughly speaking, the linear terms of $g$ as a function of $x$ can be ignored. 

\begin{proposition}
For every metric $g$ there exists a local coordinate system such that $\frac{\partial g_{ij}}{\partial x^\alpha}(0)=0$ for all $i,j,\alpha$. In particular, in this coordinate system we have $\Gamma_{ij}^k(0)=0$ and the components of the curvature tensor at $x_0=0$ are defined as some combinations of second derivatives of $g$.
\end{proposition}

The second result states that covariantly constant $(1,1)$-tensor fields $L$ are actually very simple. 
To the best of our knowledge, this theorem was first proved by A.\,P.~Shirokov  \cite{Shirokov}
(see also \cite{Boubel2, Lehmann, GThom}).

\begin{theorem}
If $L$ satisfies $\nabla L=0$  for a symmetric connection $\nabla$, then there exists a local coordinate system $x^1,\dots , x^n$ in which $L$ is constant.
\end{theorem}

In this coordinate system the equation $\nabla L=0$ can be rewritten in a very simple way:

\begin{equation}
\label{linear}
\left(\frac{\partial g_{ip}}{\partial x^\beta} - \frac{\partial g_{i\beta}}{\partial x^p} \right) L^\beta_k =
\left(\frac{\partial g_{i\beta}}{\partial x^k} - \frac{\partial g_{i k}}{\partial x^\beta} \right) L^\beta_p
\end{equation}

This equation is linear and
if we represent $g$ as a power series in $x$,  then \eqref{linear} must hold for each term of this expansion.   Moreover, if we consider the constant and second order terms only, then they will  give us a particular (local) solution.

This suggest the idea to set $L(x)=\mathrm{const}$ and then try to find the desired metric $g(x)$ in the form:
$$
\mbox{constant  +  quadratic}
$$
or
\begin{equation}
\label{eq0}
g_{ij} (x)= g_{ij}^0 + \sum  \mathcal B_{ij, pq}x^p x^q
\end{equation}
where $ \mathcal B$ satisfies obvious symmetry relations, namely,  $ \mathcal B_{ij,pq}= \mathcal B_{ji,pq}$ and $ \mathcal B_{ij,pq}= \mathcal B_{ij, qp}$.

Before discussing the explicit formula for $ \mathcal B$, we give some general remarks about ``quadratic'' metrics
 \eqref{eq0}.

\begin{itemize}

\item The condition $\nabla L=0$ amounts to  the following equation for $\mathcal B$:
\begin{equation}
\label{important1}
( \mathcal B_{ip,\beta q} -  \mathcal B_{i\beta,pq}) L^\beta_k = (  \mathcal B_{\beta i,kq} - \mathcal B_{ik,\beta q} ) L^\beta_p
\end{equation}

\item The condition that $L$ is $g$-symmetric reads:
\begin{equation}
\label{important2}
 \mathcal B_{ij,pq}L^i_l  =  \mathcal B_{il,pq}L^i_j
\end{equation}

\item The curvature tensor of $g$ at the origin $x=0$ takes the following form:
\begin{equation}
\label{important3}
R^i_{k \, \alpha \beta} =
g^{is} (  \mathcal B_{\beta s,\alpha k} +  \mathcal B_{\alpha k, \beta s} -  \mathcal B_{\beta k, \alpha s} -  \mathcal B_{\alpha s, \beta k}),
\end{equation}
and, in particular,  $R$ (at the origin) depends on $ \mathcal B$  linearly:
$$
R_{\lambda_1  \mathcal B_1 + \lambda_2  \mathcal B_2} = \lambda_1 R_{ \mathcal B_1} + \lambda_2 R_{ \mathcal B_2}
$$

\end{itemize}

\medskip

Thus,  the realisation problem admits the following purely algebraic version:  find $ \mathcal B$  satisfying \eqref{important1}, \eqref{important2} and such that
\eqref{important3} coincides with $R_{\mathrm{formal}}$ from Proposition \ref{manyblocks}. From the formal viewpoint, this is a system of linear equations on $\mathcal B$ which we need to solve or just to guess a particular solution.

\medskip

{\bf Example.}
Consider the simplest case when
$$
g = g^0 + \mathcal  B(x,x), \quad  \mathcal B_{ij}(x,x)= \sum \mathcal  B_{ij,pq}x^p x^q\quad \mbox{with } \mathcal  B=\mathcal   C\otimes \mathcal  D,
$$
where $\mathcal  C$ and $\mathcal   D$ are the bilinear forms associated with the $g^0$-symmetric operators $C$ and $D$, i.e.,  $\mathcal  B_{ij,pq}= \mathcal  C_{ij}\cdot\mathcal  D_{pq}$, 
$\mathcal  C_{ij}=g^0_{i\alpha}C^\alpha_j$, $\mathcal  D_{pq}=g^0_{p\alpha}D^\alpha_q$. Then the conditions \eqref{important1}, \eqref{important2}, \eqref{important3} can respectively be rewritten (in terms of $C$ and $D$) as
\begin{equation*}
\label{important1'}
[CXD, L] +[CXD, L]^* =0 \quad \mbox{for any $X\in \gl(V)$},
\tag{$\ref{important1}'$}
\end{equation*}

\begin{equation*}
\label{important2'}
CL = LC
\tag{$\ref{important2}'$}
\end{equation*}

\begin{equation*}
\label{important3'}
R(X) = -CXD + (CXD)^*, \qquad  X\in \so(g^0)
\tag{$\ref{important3}'$}
\end{equation*}

Similarly, if  $\mathcal B=\sum_\alpha \mathcal C_\alpha \otimes \mathcal D_\alpha$,  then the corresponding conditions on $\mathcal B$ are obtained from ($\ref{important1}'$),  ($\ref{important2}'$),  ($\ref{important3}'$) by summing over $\alpha$.

These simple observations lead us to the following conclusion. Let $B=\sum C_\alpha \otimes D_\alpha$  where $C_\alpha$ and $D_\alpha$ are $g^0$-symmetric operators.  Consider $B$ as a linear map
$$
B : \gl(V) \to \gl(V)
\quad \mbox
{defined by $B(X)=\sum C_\alpha X D_\alpha$,}   
$$

In other words,  $B(X)$ is obtained from $B$ by ``replacing'' $\otimes$ by $X$.  Then for the corresponding quadratic metric $g = g^0 + \mathcal  B(x,x)$, the conditions 
 \eqref{important1}, \eqref{important2}, \eqref{important3} can be rewritten as
\begin{equation*}
\label{important1''}
[B(X), L] +[B(X), L]^* =0 \quad \mbox{for any $X\in \gl(V)$},
\tag{$\ref{important1}''$}
\end{equation*}

\begin{equation*}
\label{important2''}
[C_\alpha, L] = 0
\tag{$\ref{important2}''$}
\end{equation*}

\begin{equation*}
\label{important3''}
R(X) = -B(X) + B(X)^*, \qquad  X\in \so(g^0)
\tag{$\ref{important3}''$}
\end{equation*} 

As the reader may notice, we prefer to work with operators rather than with forms.  We used the same idea before when we replaced $\Lambda^2(V)$ by $\so(g)$. The reason is easy to explain:  operators form an associative algebra, i.e., one can multiply them and we use this property throughout the paper.

\medskip

The last formula \eqref{important3''}, in fact,  shows  how  to reconstruct $B$ from $R(X)$:  we need to ``replace'' $X$ by $\otimes$, i.e., $B=-\frac{1}{2}R(\otimes)$.  
Namely, we consider the following formal expression:
\begin{equation}
 \label{r2}
B= -\frac{1}{2}\cdot \frac{d}{dt} \big| _{t=0} p_{\mathrm{min}}(L+t\cdot\otimes),
 \end{equation}
where $p_{\mathrm{min}}(t)$ is the minimal polynomial of $L$.  This formula looks a bit strange but,  in fact, it defines a tensor $B$ of type $(2,2)$ whose meaning is very natural. If
 $p_{min}(t) = \sum_{m=0}^n a_m t^m$ is the minimal polynomial of $L$, then 
\begin{equation}
B = - \frac{1}{2}\cdot \sum_{m=0}^n a_m \sum_{j=0}^{m-1} L^{m-1-j} \otimes  L^{j}.
\label{r3}
\end{equation}

This formula is obtained from the right hand side of \eqref{r}, i.e.,
$$
 \frac{d}{dt} \big| _{t=0}\left( \sum_{m=0}^n a_m (L+t\cdot X)^m \right)= 
\sum_{m=0}^n a_m \sum_{j=0}^{m-1} L^{m-1-j} X L^{j},
$$
by substituting $\otimes$ instead $X$.

\begin{proposition} Assume that $L$ is a $g^0$-symmetric operator and consider it as $(1,1)$-tensor field whose components are all constant  in coordinates $x$.
Define the quadratic metric $g(x)=g^0 + \mathcal B(x,x)$ with $\mathcal B_{ij,pq}=g^0_{i\alpha} g^0_{p\beta} B^{\alpha, \beta}_{j,q}$, where  $B$ is constructed from $L$ by \eqref{r2} {\rm (}or, equivalently, by \eqref{r3}{\rm )}. Then 
\medskip

$1)$ $L$ is $g$-symmetric;
\medskip

$2)$  $\nabla L=0$, where $\nabla$ is the Levi-Civita connection for $g$;
\medskip

$ 3)$ The curvature tensor for $g$ at the origin is defined by \eqref{r}, i.e.,
$$
R(X) =  \frac{d}{dt} \big| _{t=0} p_{\mathrm{min}}(L+tX).
$$

\end{proposition}

{\it Proof.}  Since $B$ is of the form $\sum_\alpha  C_\alpha \otimes D_\alpha$, where $C_\alpha$ and $D_\alpha$ are some powers of $L$, we can use formulas \eqref{important1''}, \eqref{important2''}, \eqref{important3''}  (see Example above).

Item 1)  is equivalent to \eqref{important2''} and hence  is obvious.

Next, to check 2) it suffices,  according to \eqref{important1''},  to show that
$$
[B(X),L]=0,  \quad \mbox {where } B= -\frac{1}{2}\cdot \frac{d}{dt} \big| _{t=0} p_{\mathrm{min}}(L+t\cdot X)
$$
but this has been done in Lemma \ref{lem1}.

Finally, we
compute the curvature tensor $R$ at the origin by using \eqref{important3''}:
$$
R(X)= -B(X) + B(X)^* = - 2B(X) = \frac{d}{dt} \big| _{t=0} p_{\mathrm{min}}(L+tX),
$$
as stated. Here we again use Lemma \ref{lem1} which says, in particular,  that  our $B(X)$ belongs to 
$\so(g^0)$,  i.e.,  $B(X) = - B(X)^*$.
 \qed

\medskip

This proposition together with Proposition \ref{twoblocks}  solve the realisation problem in the most important ``two Jordan blocks'' case.   To get the realisation for the general case,  we proceed just in the same way as we did for the algebraic part.  Namely, we split $L$ into Jordan blocks and  for each pair of Jordan blocks $L_i$, $L_j$ and define a formal curvature tensor $\widehat R_{ij}$  (see Section \ref{Berger} for details).  Then by using formula \eqref{r2} we can realise this formal curvature tensor by an appropriate quadratic metric $g(x) = g^0 + \widehat B_{ij}(x,x)$   satisfying $\nabla L=0$.  We omit the details because this construction is  straightforward and just repeats its algebraic counterpart discussed  in Section \ref{Berger}.  Now, if we set
$$
g(x) = g^0 + \mathcal B(x,x), \quad \mbox{with }  B=\sum_{i<j} \widehat B_{ij},
$$
then by linearity this metric still satisfies  $\nabla L=0$ and its curvature tensor coincides with $R_{\mathrm{formal}}=\sum_{i<j} \widehat R_{ij}$ from  Proposition \ref{twoblocks}.  This completes the realisation part of the proof.


\section{Appendix:  The case of a pair of complex conjugate eigenvalues}

Let $L:V\to V$ be a $g$-symmetric operator with two complex conjugate eigenvalues $\lambda$ and $\bar\lambda$.  In this case an analog of Proposition \ref{Jordan} can be formulated in complex terms.

The point is that on the vector space $V$ there is a canonical complex structure $J$ that can be uniquely defined by the following condition:    the $i$ and $-i$ eigenspaces of $J$ in $V^{\mathbb C}$  coincide with $\lambda$ and $\bar\lambda$ generalised eigenspaces of $L$ respectively.

Note that the complex structure $J$ both commutes  with $L$ and is $g$-symmetric.  This immediately implies that if we consider $V$ as a complex vector space with respect to $J$, then $L: V\to V$ is a complex operator and $g$ can be considered as the imaginary part of the following complex bilinear form
$g^{\mathbb C}:  V\times V \to \mathbb C$:
$$
g^{\mathbb C} (u,v) = g(Ju,v) + i g(u,v).
$$ 
It is easy to see that $L$ is still $g$-symmetric with respect to $g^{\mathbb C}$.

Thus, instead of looking for a real canonical form for $L$ and $g$, it is much more convenient to use a complex canonical form for $L$ and $g^{\mathbb C}$.  As a complex operator, $L$  has a single eigenvalue $\lambda$ and therefore we are lead to the situation described in Proposition \ref{Jordan}.  Replacing $\mathbb R$ by $\mathbb C$ does not change the conclusion:
there exists a complex coordinate system such that  $L$ and $g^{\mathbb C}$ are given exactly by the same matrices as $L$ and $g$ are in Proposition \ref {Jordan}.

In this canonical complex coordinate system, the statement of Proposition \ref{description}  remains unchanged if we replace the real Lie algebra $\goth{so}(g)$ by the complex Lie algebra $\goth{so}\bigl(g^{\mathbb C}\bigr)$  (the entries of all matrices in \eqref{X}--\eqref{Mij} are now, of course, complex).  This two Lie algebras are different, but we have the obvious inclusion
$\goth{so}\bigl(g^{\mathbb C}\bigr) \subset \goth{so}(g)$.  It is also important that $\goth g_L$ turns out to be a complex Lie algebra,  i.e., $\goth g_L \subset \goth{so}\bigl(g^{\mathbb C}\bigr)$.

To show that $\goth g_L$ is still Berger in this case, we first  need to verify the conclusion  of Propositions \ref{twoblocks}, i.e., to check that the image of the operator \eqref{r}  coincides with 
$\goth g_L$.  

Proposition \ref{twoblocks} is purely algebraic, so it remains true for a complex operator $L$ and a complex bilinear form $\goth g^{\mathbb C}$, if we define $R: \goth{so}(g^{\mathbb C}) \to 
\goth g_L$ by  \eqref{r} with $p_{\mathrm{min}}(t)=(t-\lambda)^n$.

We now must take care of two issues.  First of all,  $R$ should be defined on a larger Lie algebra, namely on $\goth{so}(g)$.   Second,   instead of $(t-\lambda)^n$ we should consider the real minimal polynomial  $p_{\mathrm{min}}(t)=(t-\lambda)^n (t-\bar\lambda)^n$  (otherwise, $R$ won't be real!).

The first issue is not much trouble at all:  we can restrict $R$ on the subalgebra $\goth{so}(g^{\mathbb C})\subset \goth{so}(g)$ and if the image still coincides with $\goth g_L$,  then the same will be true for the original operator (we use the fact that the image of $R$ belongs to $\goth g_L$ automatically,  Lemma~\ref{lem1}).

To sort out the second problem, we simply compute $R$  for the minimal polynomial $p_{\mathrm{min}}(t)=(t-\lambda)^n (t-\bar\lambda)^n$ thinking of $L$ and $X\in \goth{so}(g^{\mathbb C})$ as complex operators and using the fact that $(L-\lambda)^n=0$: 
$$
\begin{aligned}
&R(X)= \frac{d}{dt}\big|_{t=0} \Bigl((L-\lambda+tX)^n\cdot (L-\bar\lambda+tX)^n\Bigr) =\\
 &\left( \frac{d}{dt}\big|_{t=0} (L-\lambda+tX)^k \right)\cdot (L-\bar\lambda)^n +
(L-\lambda)^n\cdot  \frac{d}{dt}\big|_{t=0} (L-\bar\lambda+tX)^n=\\
&\left( \frac{d}{dt}\big|_{t=0} (L-\lambda+tX)^n \right)\cdot (L-\bar\lambda)^n.\end{aligned}
$$

The operator in the first bracket is the same as in Proposition \ref{twoblocks}.  In particular, its image coincides with $\goth g_L$, as needed.  After this we  multiply the result by the non-degenerate matrix $(L-\bar\lambda)^k$. This operation cannot change the dimension of the image, and since we know that $\mathrm{Im}\, R$ is contained in $\goth g_L$ automatically (Lemma~\ref{lem1}), we conclude  that $\mathrm{Im}\, R = \goth g_L$.

The proof of Proposition \ref{manyblocks}  does not use any specific property of the ``small'' operators $\widehat R_{ij}$. We only need the image of $\widehat R_{ij}$  to coincide with the subalgebra $\goth m_{ij} \subset \goth g_L$.   But this is exactly the statement of Proposition \ref{twoblocks} which still holds true in the case of two complex blocks.

Thus,  if $L$ has two complex conjugate eigenvalues $\lambda$ and $\bar\lambda$,  the Lie algebra  $\goth g_L$ is still Berger.


\begin{thebibliography}{99}





\bibitem{Alexeevski} 
D.~V.~Alekseevskii, 
`Riemannian spaces with unusual holonomy groups', 
Funct. Anal. Appl. 2 (1968) 97--105.


\bibitem{Ambrose} 
W.~Ambrose and I.~M.~Singer, 
`A theorem on holonomy', 
Trans. Amer. Math. Soc.  75 (1953) 428--443.


\bibitem{Aminova1} 
A.~V.~Aminova,
`Pseudo-Riemannian manifolds with general geodesics',
Russian Math. Surveys  48  (1993),  no. 2, 105--160.


\bibitem{Aminova2}
 A.~V.~Aminova, 
 `Projective transformations of pseudo-Riemannian manifolds. 
 Geometry, 9. J. Math. Sci. (N. Y.) 113 (2003), no. 3, 367--470.


\bibitem{Aminovabook} 
A.~V.~Aminova,
{\it Projective transformations of pseudo-Riemannian manifolds}.
Moscow, Yanus-K, 2003 (Russian).


\bibitem{Ber} 
M.~Berger, 
`Sur les groupes d'holonomie des vari\'et\'es a connexion affine et des 
vari\'et\'es Riemanniennes', 
Bull. Soc. Math. France 83 (1955) 279--330.


\bibitem{Bergery}
L.~B\'erard Bergery, A.~Ikemakhen, 
`Sur l'holonomie des vari\'et\'es pseudo-riemanniennes
de signature $(n; n)$'. Bull. Soc. Math. France Vol. 125 (1997), 93--114.


\bibitem{splitting}
A.~V.~Bolsinov and V.~S.~Matveev,
`Splitting and gluing lemmas for geodesically equivalent pseudo-Riemannian metrics',
Trans. Amer. Math. Soc. 363 (2011), 4081--4107. 




\bibitem{locnf}
A.~V.~Bolsinov and V.~S.~Matveev,
`Local normal forms for geodesically equivalent pseudo-Riemannian metrics',  arXiv:1301.2492
(submitted to Trans. Amer. Math. Soc). 






\bibitem{BolsMatvKios} 
A.~V.~Bolsinov, V.~Kiosak and V.~S.~Matveev, 
`A Fubini theorem for pseudo-Riemannian geodesically equivalent metrics', 
J. London Math. Soc. (2) 80 (2009) 341--356.


\bibitem{Boubel2} 
C.~Boubel, 
`An integrability condition for fields of nilpotent endomorphisms',
arXiv:1003.0979.


\bibitem{Br1} 
R.~Bryant, 
`A survey of Riemannian metrics with special holonomy groups', 
Proc. ICM Berkeley, Amer. Math. Soc., 505--514 (1987).


\bibitem{Br2} 
R.~Bryant, 
`Metrics with exceptional holonomy', 
Ann. Math. 126, 525-576 (1987).


\bibitem{Br4} 
R.~Bryant, 
`Classical, exceptional, and exotic holonomies: A status report'. 
Besse, A. L. (ed.), Actes de la table ronde de g\'eom\'etrie diff\'erentielle en l'honneur de Marcel Berger, Luminy, France, 12--18 juillet 1992. Soc. Math. France. S\'emin. Congr. 1, 93--165 (1996).


\bibitem{Ca1} 
\'E.~Cartan, 
`Les groupes d'holonomie des espaces g\'en\'eralis\'es', 
Acta.Math. 48, 1--42 (1926) ou Oeuvres compl\`etes, tome III, vol. 2, 997--1038.


\bibitem{Ca2}
\'E.~Cartan,
`Sur une classe remarquable d'espaces de Riemann', 
Bull. Soc. Math. France 54, 214--264(1926), 55, 114--134 (1927) 
ou Oeuvres compl\`etes, tome I, vol. 2 , 587-659.


\bibitem{TrofFom}
A.~T.~Fomenko and V.~V.~Trofimov,
{\it Integrable systems on Lie Algebras and Symmetric Spaces} 
(Gordon and Breach,  London/New York, 1988).


\bibitem{Gal1}
A.S.~Galaev, 
`Classification of connected holonomy groups of pseudo-K\"ahlerian manifolds of index 2', arXiv:math.DG/0405098v2, 2005.


\bibitem{Gal3}
A.S.~Galaev, 
`Remark on holonomy groups of pseudo-Riemannian manifolds of signature $(2; n + 2)$, arXiv:math.DG/0406397, 2004.


\bibitem{Gal2} 
A.~Galaev, 
`Metrics that realize all Lorentzian holonomy algebras'. 
Int. J. Geom. Methods Mod. Phys.,  3 (2006), no. 5\&6, p. 1025--1045.


\bibitem{GalaevLeistner} 
A.~S.~Galaev and T.~Leistner,
`Recent developments in pseudo-Riemannian holonomy theory'. Handbook of pseudo-Riemannian geometry and supersymmetry, 581--627, IRMA Lect. Math. Theor. Phys., 16, Eur. Math. Soc., Z\"urich, 2010. 


\bibitem{Ike}
A.~Ikemakhen, 
`Sur l'holonomie des vari\'et\'es pseudo-riemanniennes de signature $(2; 2+
n)$' , Publ. Mat. 43 (1999), no. 1, 55--84.


\bibitem{Joy1}
D.~Joyce, 
`Compact Riemannian 7-manifolds with holonomy $G_2$. I', 
Journal of Differential Geometry 43 (1996), 291--328.


\bibitem{Joy2}
D.~Joyce, 
`Compact Riemannian 7-manifolds with holonomy $G_2$. II', 
Journal of Differential Geometry 43 (1996), 329--375.


\bibitem{Joy3}
D.~Joyce, 
`A new construction of compact 8-manifolds with holonomy $\mathrm{Spin}(7)$', 
Journal of Differential Geometry 53 (1999), 89--130. 


\bibitem{KM1}
V.~Kiosak, V.~S.~Matveev, `Complete Einstein metrics are geodesically rigid',
Comm. Math. Phys. 289(2009), no. 1, 383--400.


\bibitem{KM2}
V.~Kiosak, V.~S.~Matveev, 
`Proof of projective Lichnerowicz conjecture for
pseudo-Riemannian metrics with degree of mobility greater than two', 
Comm. Mat. Phys. 297(2010), 401--426.

\bibitem{Lancaster}
P.~Lancaster,  and L.~Rodman, 
`Canonical forms for hermitian matrix pairs under strict equivalence and congruence', 
SIAM Review, Vol. 47, 2005, 407--443.


\bibitem{Lehmann}
J.~Lehmann-Lejeune,
`Integrabilit\'e des $G$-structures d\'efinies par une 1-forme 0-d\'eformable
\`a valeurs dans le fibr\'e tangent'. 
Ann. Inst. Fourier 16 (1966), 329--387.


\bibitem{Le1}
T.~Leistner, 
`On the classification of Lorentzian holonomy groups'. J. Differential Geom. 76 (2007), no. 3, 423--484.


\bibitem{LC}
T.~Levi-Civita, 
`Sulle transformazioni delle equazioni dinamiche', 
Ann. Mat. ($2^a$) 24 (1896) 255--300.


\bibitem{Manakov} 
S.~V.~Manakov, 
`Note on the integration of Euler's equation of the dynamics of an N-dimensional rigid body', 
Funct. Anal. Appl. 11 (1976) 328--329.


\bibitem{Matveev} 
V.~S.~Matveev, 
`Hyperbolic manifolds are geodesically rigid', 
Invent. Math. 151(2003), 579--609.


\bibitem{MerkSchw} 
S.~Merkulov and L.~Schwachh\"{o}fer, 
`Classification of irreducible holonomies of torsion-free affine connections', 
Ann. Math. 150 (1999) 77--149.


\bibitem{Mikes}
J.~Mikes, 
`Geodesic mappings of affine-connected and Riemannian spaces'. 
Geometry, 2, J. Math. Sci. 78 (1996), no. 3, 311--333.


\bibitem{MischFom} 
A.~S.~Mischenko and A.~T.~Fomenko , 
`Euler equations on finite dimensional Lie groups', 
Izv. Akad. Nauk SSSR Ser. Mat. 42 (1978) 396--415 (Russian); Math. USSR--Izv. 12 (1978) 371--389 (English)


\bibitem{Sch2} 
L.~Schwachh\"ofer, 
`Connections with irreducible holonomy representations'. 
Adv.~Math. 160 (2001), no.~1, 1--80.


\bibitem{Shirokov}
A.~P.~Shirokov, 
`On a property of covariantly constant affinors'.  
Dokl. Akad. Nauk SSSR (N.S.) 102 (1955), 461--464 (Russian).


\bibitem{Sinjukov}
N.~S.~Sinjukov,
{\it Geodesic mappings of Riemannian spaces}  
Nauka, Moscow, 1979  (Russian).


\bibitem{Solodovnikov}
G.\,I.~Kru\v{c}kovi\v{c}, A.\,S.~Solodovnikov, 
`Constant symmetric tensors in Riemannian spaces'. 
Izv. Vys\v{s}. U\v{c}ebn. Zaved. Matematika 1959 no. 3 (10), 147--158 (Russian).


\bibitem{GThom}
G.~Thompson,
`The integrability of a field of endomorphisms'.
Mathematica Bohemica, Vol. 127 (2002), No. 4, 605--611.


\bibitem{Thompson}
R.\,C.~Thompson, 
`Pencils of complex and real symmetric and skew matrices'.
Linear Algebra and its Appl.,  147 (1991),  323--371.



\bibitem{Yau}
S.\,T.~Yau, `On the Ricci curvature of a compact K\"ahler manifold and the complex Monge-Amp\`ere equation. I'. Comm.  Pure Appl. Math., 31 (1978), 339--411.




\end{thebibliography}
\end{document}